\documentclass[reqno]{amsart}
\usepackage{amssymb}
\usepackage{graphics}

\def\Vecc{\mathrm{Vec}}
\def\Dend{\mathrm{Dend}}
\def\Var{\mathrm{Var}}
\def\Perm{\mathrm{Perm}}
\def\Comm{\mathrm{Comm}}
\def\CommTrias{\mathrm{CommTrias}}
\def\sCommTrias{\mathrm{sCommTrias}}
\def\Dias{\mathrm{Dias}}

\def\PreLie{\mathrm{PreLie}}
\def\PostLie{\mathrm{PostLie}}
\def\dend{\mathrm{dend}\,}

\newtheorem{thm}{Theorem}
\newtheorem{lem}{Lemma}
\newtheorem{remark}{Remark}
\newtheorem{prop}{Proposition}
\newtheorem{cor}{Corollary}

\theoremstyle{definition}
\newtheorem{defn}{Definition}
\newtheorem{exmp}{Example}

\title{On embedding of dendriform algebras into Rota---Baxter algebras}

\author{V.~Yu.~Gubarev$^{1)}$, P.~S.~Kolesnikov$^{2)}$}

\address{$^{1)}$Novosibirsk State University, Novosibirsk, Russia}
\email{vsevolodgu@mail.ru}

\address{$^{2)}$Sobolev Institute of Mathematics, Novosibirsk, Russia}
\email{pavelsk@math.nsc.ru}

\begin{document}

\begin{abstract}
Following a recent work by C.~Bai, O.~Bellier, L.~Guo, X.~Ni (arXiv:1106.6080)
we define what is a dendriform di- or trialgebra in an arbitrary variety Var of binary algebras
(associative, commutative, Poisson, etc.).
We prove that every dendriform dialgebra in Var can be embedded into
a Rota---Baxter algebra of weight zero in the same variety,
and every dendriform trialgebra can be embedded into a Rota---Baxter algebra of nonzero weight.
\end{abstract}

\maketitle

\section{Introduction}

In 1960 Glen Baxter  \cite{Baxter60}
introduced an identity defining what is now called Rota---Baxter
operator in developing works of
F.~Spitzer \cite{Spitz56} in fluctuation theory.
By definition,
a Rota---Baxter operator $R$ of weight $\lambda $ on an algebra $A$
is a linear map on $A$ such that
\[
R(x)R(y) = R(xR(y) + R(x)y) + \lambda R(xy), \quad x,y\in A,
\]
where $\lambda $ is a scalar from the base field.

Later, commutative associative
algebras with such an operator were studied by G.-C.~Rota and others
\cite{Cartier72, Rota68}.
In 1980s, these operators appeared in the context
of Lie algebras independently in works A.~A.~Belavin and V.~G.~Drinfeld
\cite{BelaDrin82} and M.~A.~Semenov-Tian-Shansky \cite{Semenov83} in
research of solutions of classical Young---Baxter equation named in the honour
of physicists Chen Ning Yang and Rodney Baxter.

For the present time, numerous connections of Rota---Baxter
operators with other areas of mathematics are found.
The latter include quantum field theory, Young---Baxter equations, operads,
Hopf algebras, number theory etc. \cite{Hopf03, QFT00, QFT01, NumThe06, Leroux04}.

The notion of a Leibniz algebra
introduced by J.-L.~Loday \cite{Loday93} is originated from
cohomology theory of Lie algebras; this is a noncommutative
analogue of Lie algebras.
Associative dialgebras (now often called diassociative algebras)
emerged in the paper by J.-L.~Loday and T.~Pirashvili \cite{LP93},
they play the role of universal enveloping associative algebras
for Leibniz algebras.
Dendriform dialgebras were defined by J.-L.~Loday in 1999 \cite{Dialg99}
in his study of algebraic $K$-theory.
Moreover, they occur to be Koszul-dual to diassociative algebras.
In 2001, J.-L.~Loday and V.~Ronco  \cite{Trialg01}
introduced a generalization of dialgebras---trialgebras
and dual to them dendriform trialgebras.

M.~Aguiar in 2000 \cite{Aguiar00} was the first who noticed a
relation between Rota---Baxter algebras and
dendriform algebras.
He proved that an associative algebra
with a Rota---Baxter operator $R$
of weight zero relative to operations $a\prec b=aR(b)$, $a\succ b=R(a)b$
is a dendriform dialgebra. In 2002, K.~Ebrahimi-Fard \cite{Fard02} generalized
this fact to the case of Rota---Baxter algebras of arbitrary weight and obtained
as result both dendriform dialgebra and dendriform trialgebra.
In the paper by K.~Ebrahimi-Fard and L.~Guo \cite{FardGuo07} in 2007,
universal enveloping Rota---Baxter algebras of weight $\lambda$
for dendriform dialgebras and trialgebras were defined.

The natural question: Whether an arbitrary dendriform di- or trialgebra
can be embedded into its universal enveloping Rota---Baxter algebra
was solved positively in \cite{FardGuo07} for free dendriform algebras
only.
In 2010, Y.~Chen and Q.~Mo
proved that any dendriform dialgebra over a field of characteristic zero
can be embedded into an appropriate Rota---Baxter
algebra of weight zero~\cite{Chen11}
using the Gr\"obner---Shirshov bases technique for Rota---Baxter algebras
developed in \cite{BokutChen}.

To solve the problem for any dendriform dialgebra (or trialgebra)
from a Rota---Baxter algebra of arbitrary weight, C.~Bai, L.~Guo and K.~Ni
\cite{BaiGuoNi10} introduced in 2010  a notion of
$\mathcal{O}$-operators,
a generalization of Rota---Baxter operators and proved that
every dendriform di- or trialgebra can be explicitly obtained
from an algebra with a $\mathcal{O}$-operator.

In a recent work \cite{BBGN2011}, the results of Aguiar and Ebrahimi-Fard
were extended to the case of arbitrary operad of Rota---Baxter algebras
and dendriform dialgebras and trialgebras.

In the present work, we solve the following problem.
Given a binary operad $\mathcal P_{\Var }$ governing a variety $\Var $
of $\Omega $-algebras ($\Omega $ is a set of binary operations),
we define what is a di- or tri-$\Var $-dendriform algebra (following \cite{BBGN2011}).
Then we construct a Rota---Baxter $\Omega $-algebra from the variety
$\Var $ such that the initial dendriform di- or trialgebra
embeds into this Rota---Baxter algebra in the sense of
Aguiar and Ebrahimi-Fard (for trialgebras, we demand $\lambda\neq0$).

The idea of the construction can be easily illustrated as follows.
Suppose $(A,\prec, \succ ,\cdot )$ is an (associative) dendriform
trialgebra. Then the direct sum of two isomorphic copies of $A$,
the space $\hat A = A\oplus A'$, endowed with a binary operation
\[
a*b = a\prec b + a\succ b + a\cdot b,
\quad
a*b' = (a\succ b)', \quad
a'*b = (a\prec b)', \quad
a'*b' = (a\cdot b)'
\]
for $a,b\in A$, is an associative algebra.
Moreover, the map $R(a')=a$, $R(a)=-a$ is a Rota---Baxter operator
of weight 1 on $\hat A$. The embedding of $A$ into $\hat A$ is given
by $a\mapsto a'$, $a\in A$.

In this work, we also introduce and consider some modification of Loday's notion of
trialgebras which we will call skew trialgebras (or s-trialgebras, for short).
This class of algebras appears from differential and
$\mathbb{Z}$-conformal algebras. Associative skew trialgebras
turn to be related with a natural noncommutative analogue
of Poisson algebras. Dendriform s-trialgebras are
Koszul dual to s-trialgebras
and they are also connected with Rota---Baxter algebras
in the same way as usual dendriform dialgebras and trialgebras.

\section{Operads for di- and trialgebras}

Our main object of study is the class of dendriform di- or tri-algebras.
In this section, we start with objects from the ``dual world'' in the sense
of Koszul duality.

The notion of an operad once introduced in \cite{May} has been
reincarnated in the beginning of 2000s. We address the reader
to either of perfect expositions of this notion and its
applications in universal algebra, e.g.,
\cite{GinzKapr, Leinster, Stasheff}.

Throughout this paper, $\Bbbk $ is an arbitrary base field.
All operads are assumed to be families of linear spaces, compositions
are linear maps, and the actions of symmetric groups are also linear.

By an $\Omega $-algebra we mean a linear space equipped with
a family of binary linear operations $\Omega = \{ \circ_i \mid i\in I\}$.
Denote by $\mathcal F$ the free operad governing the variety of all
$\Omega $-algebras. For every natural $n>1$, the space
$\mathcal F(n)$ can be identified with the space spanned by all binary trees
with $n$ leaves
labeled by $x_1,\dots, x_n$, where each vertex (which is not a leaf) has a label
from $\Omega $.

Let $\Var $ be a variety of $\Omega $-algebras defined by a family
$S$ of poly-linear identities of any degree
(which is greater than one).
Denote by $\mathcal P_{\Var} $ the binary operad governing the
variety $\Var$, i.e., every algebra from $\Var $  is a functor
from $\mathcal P_{\Var}$ to the multi-category $\Vecc$ of linear spaces
with poly-linear maps.

Denote by
$\Omega^{(2)}$ and $\Omega ^{(3)}$
the sets of binary operations
$\{\vdash_i, \dashv_i \mid i\in I \}$ and
 $\Omega^{(2)}\cup \{\perp_i \mid i\in I \}$, respectively.
Similarly, let $\mathcal F^{(2)}$ and $\mathcal F^{(3)}$
stand for the free operads governing the varieties of
all $\Omega^{(2)}$- and $\Omega^{(3)}$-algebras, respectively.

We will need the following important operads.

\begin{exmp}
Operad $\Perm$ introduced in \cite{Chapoton}
is governing the variety of Perm-algebras \cite[p.~17]{Zinbiel}.
Namely, $\Perm(n) =\Bbbk^n$ with a standard basis
$e_i^{(n)}$, $i=1,\dots, n$. Every $e^{(n)}_i$ can be identified
with an associative and commutative poly-linear monomial in $x_1,\dots, x_n$
with one emphasized variable~$x_i$.
\end{exmp}

\begin{exmp}
Operad $\CommTrias$ introduced in \cite{Vallette2007}
is governing the variety of associative and commutative
trialgebras, see \cite[p.~25]{Zinbiel}.
Namely, $\CommTrias(n)$ has a standard basis
$e_H^{(n)}$, where $\emptyset \ne H\subseteq \{1,\dots, n\}$.
Such an element (corolla) can be identified with
a commutative and associative monomial with several
emphasized variables $x_j$, $j\in H$.
\end{exmp}

The number of observations made, for example, in
\cite{Vallette2008, Chapoton, Kol2008} lead to the following natural
definition.

\begin{defn}
A {\em di-$\Var $-algebra\/} is a functor from
$\mathcal P_{\Var}\otimes \Perm $ to $\Vecc $, i.e., an
$\Omega^{(2)}$-algebra satisfying the following identities:
\begin{gather}
(x_1\dashv_i x_2)\vdash_j x_3 = (x_1\vdash_i x_2)\vdash_j x_3, \quad
x_1\dashv_i (x_2 \vdash_j x_3) = x_1\dashv_i (x_2 \dashv_j x_3),
                                                    \label{eq:0-identities} \\
f(x_1,\dots , \dot x_k, \dots , x_n), \quad f\in S, \ n=\deg f, \ k=1,\dots, n,
                                                    \label{eq:Dot-identities}
\end{gather}
where
$i,j\in I$, and
$f(x_1,\dots , \dot x_k, \dots , x_n)$ stands for $\Omega^{(2)}$-identity
obtained from
$f$ by means of replacing all products $\circ_i$
with either $\dashv_i$ or $\vdash_i$
in such a way that all horizontal dashes point to the selected variable~$x_k$.
\end{defn}

\begin{exmp}
Let $|\Omega |=1$.
The variety of diassociative algebras (or associative dialgebras)
\cite{LP93} is given by
\eqref{eq:0-identities} together with
\begin{equation}\label{eq:Diass}
\begin{gathered}
x_1\dashv (x_2\dashv x_3) = (x_1\dashv x_2)\dashv x_3, \quad
x_1\vdash (x_2\dashv x_3) = (x_1\vdash x_2)\dashv x_3, \\
x_1\vdash (x_2\vdash x_3) = (x_1\vdash x_2)\vdash x_3.
\end{gathered}
\end{equation}
\end{exmp}

\begin{exmp}
Consider the class of Poisson algebras ($|\Omega |=2$), where $\circ_1$ is an
associative and commutative product (we will denote $x\circ_1 y $ simply by $xy$)
and $\circ_2$ is a Lie product ($x\circ_2 y = [x,y]$) related with $\circ_1$ by
the following identity:
\[
 [x_1x_2, x_3] = [x_1,x_3]x_2 + x_1[x_2,x_3].
\]
Then a di-Poisson algebra is a linear space equipped by
four operations
$(\cdot * \cdot)$, $[\cdot * \cdot ]$, $*\in \{\vdash, \dashv\}$
satisfying \eqref{eq:0-identities} and  \eqref{eq:Dot-identities}.
Commutativity of the first product and anticommutativity
of the second one allow to reduce these four operations to
only two, since  \eqref{eq:Dot-identities} implies
\[
(x_1\dashv x_2) = (x_2\vdash x_1), \quad
[x_1\dashv x_2] = - [x_2\vdash x_1].
\]
With respect to the operations
\[
xy :=(x\vdash y), \quad [x,y]:=[x\vdash y],
\]
the identities \eqref{eq:0-identities} and  \eqref{eq:Dot-identities}
 are equivalent to the following system:
\[
\begin{gathered}
x_1 (x_2 x_3) = (x_1 x_2) x_3, \quad
(x_1x_2)x_3 = (x_2x_1)x_3, \\
[x_1,[x_2,x_3]] - [x_2,[x_1,x_3]] = [[x_1,x_2],x_3], \\
[x_1x_2, x_3] =  x_1[x_2,x_3] + x_2 [x_1,x_3], \\
[x_1, x_2x_3] = [x_1,x_2]x_3 + x_2[x_1,x_3].
\end{gathered}
\]
In \cite{Dialg99}, a more general class was introduced (without
assuming commutativity of the associative product).
\end{exmp}

A similar approach works for trialgebras.

\begin{defn}\label{defn:TriVarAlgebra}
A {\em tri-$\Var $-algebra\/} is a functor from
$\mathcal P_{\Var}\otimes \CommTrias $ to $\Vecc $, i.e., an
$\Omega^{(3)}$-algebra satisfying the following identities:
\begin{equation}\label{eq:tri0-identities}
\begin{gathered}
(x_1 \ast_i x_2)\vdash_j x_3 = (x_1\vdash_i x_2)\vdash_j x_3, \quad
x_1\dashv_i (x_2 \ast_j x_3) = x_1\dashv_i (x_2 \dashv_j x_3),\\
\ \ast\in \{\vdash, \dashv, \perp\},
\end{gathered}
\end{equation}
\begin{equation}                    \label{eq:triDot-Identities}
\begin{gathered}
f(x_1,\dots , \dot x_{k_1}, \dots , \dot x_{k_l}, \dots, x_n),
\\
 f\in S, \ n=\deg f, \ 1\le k_1<\dots <k_l\le n,
\ l=1,\dots, n. \nonumber
\end{gathered}
\end{equation}
where
$i,j\in I$,  and
$f(x_1,\dots , \dot x_{k_1}, \dots , \dot x_{k_l}, \dots, x_n)$
is the result of a procedure described below.
(It is somewhat similar to the tri-successor procedure from \cite{BBGN2011}).
\end{defn}

Suppose $u=  u(x_1,\dots, x_n)\in \mathcal F(n)$ is a non-associative $\Omega$-monomial.
Fix $l$ indices $1\le k_1<\dots <k_l\le n$, and
denote the the monomial $u$ with $l$ emphasized variables
$x_{k_j}$, $j=1,\dots, l$,
by $u^H$, $H=\{k_1,\dots, k_l \}$.
Now, identify  $u^H$
 with an element from
$\mathcal F(n)\otimes \CommTrias(n)$ in the natural way:
\[
u^H \equiv  u\otimes e^{(n)}_{k_1,\dots, k_l}.
\]
It can be considered as a binary tree from $\mathcal F(n)$
with $l$ emphasized leaves.

\begin{center}
\medskip
\includegraphics{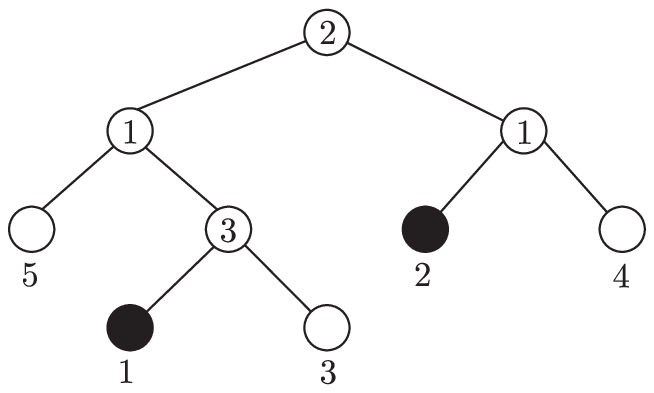}  \\
Fig. 1\\
Example: $u^H= (x_5\circ_1 (\dot x_1\circ_3 x_3))\circ_2 (\dot x_2\circ_1 x_4)$,
$H=\{1,2\}$.
Emphasized leaves
are colored in black, others---in white.
\medskip
\end{center}

Now the task is to mark all vertices of $u^H$ with
appropriate labels from $\Omega^{(3)}$.
Define the family of maps
\[
\Phi(n) : \mathcal F(n)\otimes \CommTrias(n) \to \mathcal F^{(3)}(n), \quad
 n\ge 1,
\]
as follows.
Given
$u\otimes e^{(n)}_{k_1,\dots, k_l}\in \mathcal F(n)\otimes \CommTrias(n)$,
the structure of the tree $u$ as well as labels of leaves do
not change. For $n=1$, there is nothing to do.
If  $u=v\circ_i w$ then
the set $H=\{k_1,\dots, k_l\}$
 of emphasized variables splits into two subsets,
$H=H_1\dot\cup H_2$, where $H_1$ consists of all
$k_j$ such that $x_{k_j}$ appears in~$v$.
Assume $\deg v = p$, then $\deg w=n-p$.
Set
\[
\Phi(n)(u^H) = \begin{cases}
 \Phi(p)(v^{H_1}) \perp_i  \Phi(n-p)(w^{H_2}), & \mbox{if } H_1,H_2\ne \emptyset, \\
 v^{\vdash}\vdash_i \Phi(n-p)(w^H), & \mbox{if } H_1=\emptyset, \\
 \Phi(p)(v^H)\dashv_i w^{\dashv}, & \mbox{if } H_2=\emptyset,
 \end{cases}
\]
where $v^\vdash$  (or $w^{\dashv} $) stands for the tree where each vertex label $\circ_j$
turns into $\vdash_j $ ($\dashv_j$).

One may extend $\Phi(n)$ by linearity, so, if
$f(x_1,\dots, x_n) = \sum\limits_{\xi }\alpha_\xi u_\xi \in \mathcal F(n)$
then
\[
f(x_1,\dots , \dot x_{k_1}, \dots , \dot x_{k_l}, \dots, x_n):=
\sum\limits_{\xi }\alpha_\xi  \Phi(n)(u^H_\xi).
\]

\begin{center}
\medskip
\includegraphics{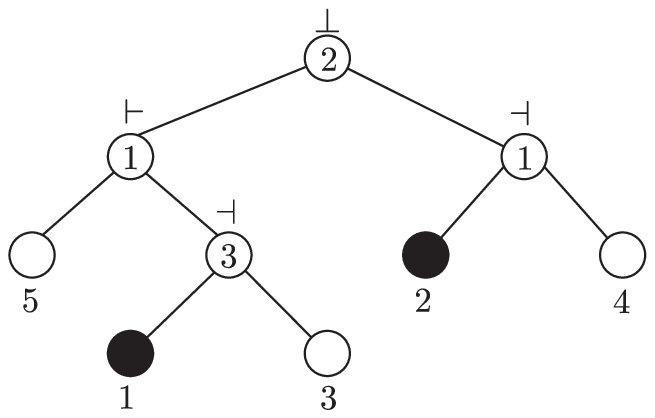}  \\
Fig. 2\\
Example: $\Phi(5)(u^H)= (x_5\vdash_1 (x_1\dashv_3 x_3))\perp_2 (x_2\dashv_1 x_4)$
for $u$ and $H$ as on Fig.~1.   \\
For each vertex which is not a leaf we assign $\perp $ if both left and right
branches have emphasized leaves. If only left branch contains an emphasized leaf
then we assign $\dashv $ to this vertex and to all vertices of the right branch.
Symmetrically, if only right branch contains an emphasized leaf then we assign
$\vdash $ to this vertex and to all vertices of the left branch.
\medskip
\end{center}

\begin{exmp}[Tri-associative algebra]\label{exmp:Triass}
Let $\Var =\mathrm{As}$ be the variety of associative algebras.
It has only one defining identity
$(x_1,x_2,x_3)=x_1(x_2x_3) - (x_1x_2)x_3$,
that turns into seven identities
\eqref{eq:triDot-Identities}. Indeed,
each nonempty subset $H\subseteq \{1,2,3\}$
gives rise to an identity of $\Omega^{(3)}$-algebras,
$\Omega^{(3)}=\{\vdash, \dashv, \perp\}$.
If $|H|=1$ then these are just the identities of a di-$\mathrm{As}$-algebra
 \eqref{eq:Diass}.
For $|H|=2$, we obtain three identities, e.g., if
$H=\{1,3\}$ then the corresponding identity
$(\dot x_1, x_2, \dot x_3)$ is
$x_1\perp (x_2 \vdash x_3)- (x_1\dashv x_2)\perp x_3$
If $H=\{1,2,3\}$ then we obtain the relation of associativity for $\perp $.
Together with \eqref{eq:tri0-identities}, these are exactly the
defining identities of triassociative algebras \cite[p.~23]{Zinbiel}.
\end{exmp}

\begin{exmp}
Let $A$ be an associative algebra.
Then the space $A^{\otimes3}$ with respect to operations
\begin{gather*}
a\otimes b\otimes c\vdash a'\otimes b'\otimes c'=abca'\otimes b'\otimes c',\quad
a\otimes b\otimes c\dashv a'\otimes b'\otimes c'=a\otimes b\otimes ca'b'c',\\
a\otimes b\otimes c\perp a'\otimes b'\otimes c'=a\otimes bca'b'\otimes c'
\end{gather*}
is a triassociative algebra.
\end{exmp}

The following construction invented in \cite{Pozh2009} for dialgebras
also works for trialgebras.

Let $A$ be a $0$-trialgebra, i.e., an $\Omega^{(3)}$-algebra which
satisfies \eqref{eq:tri0-identities}.
Then
$A_0=\mathrm{Span}\,\{a\vdash_i b-a\dashv_i b,\,
a\vdash_i b-a\perp_i b\mid a,b\in A,\, i\in I\}$
is an ideal of $A$. The quotient $\bar A = A/A_0$ carries
a natural structure of an $\Omega $-algebra.
Consider the formal direct sum $\hat A = \bar A\oplus A$ with (well-defined) operations
\begin{equation}\label{eq:EilenbergDialg}
\bar{a}\circ_i x=a\vdash_i x,\quad x\circ_i\bar{a}=x\dashv_i a,\quad \bar{a}\circ_i\bar{b}
=\overline{a\vdash_i b},\quad x\circ_i y=x\perp_i y,
\end{equation}
$\bar a, \bar b \in \bar A$, $x,y\in A$.

\begin{prop}\label{prop:EilenbergDouble}
A trialgebra $A$ satisfying \eqref{eq:tri0-identities} is a tri-$\Var$-algebra if and only if $\hat{A}$
is an algebra from the variety $\Var$.
\end{prop}

\begin{proof}
The claim follows from the following observation.
If $f(x_1,\dots, x_n)\in \mathcal F(n)$ then
the value $f(\bar a_1, \dots, \bar a_n)$ in $\bar A\subset \hat A$
is just the image of $[\Phi(n)(f^H)](a_1, \dots, a_n)$ in $\bar A$
for any subset $H$; moreover,
the value of $f(x_1, \dots, \dot x_{k_1}, \dots, \dot x_{k_l}, \dots , x_n)$
on $a_1,\dots, a_n\in A$ is equal to
$f(\bar a_1, \dots, a_{k_1}, \dots, a_{k_l}, \dots, \bar a_{n})\in \hat A$,
i.e., one has to add bars to all non-emphasized variables.
\end{proof}

Assuming $x\perp_i y\equiv 0$ for all
$x,y\in A$, $i\in I$,  we obtain the construction from
\cite{Pozh2009}. This construction turns to be
useful in the study of dialgebras (see, e.g., \cite{GubKol, Voronin}).
For a variety $\Var $,
let us denote by $\mathcal D_{\Var}$ and $\mathcal{T}_{\Var} $
the operads governing di- and tri-$\Var $-algebras, respectively.

The structure of a di-$\Var $-algebra may be recovered from a
structure of a $\Var$-pseudo-algebra over an appropriate
bialgebra $H$. Let us recall this notion from \cite{BDK2001}.
Suppose $H$ is a cocommutative bialgebra with a coproduct $\Delta $ and counit
$\varepsilon $. We will use the Swedler notation for $\Delta$, e.g.,
$\Delta(h)=h_{(1)}\otimes h_{(2)}$,
$\Delta^{2}(h):=(\Delta \otimes \mathrm{id})\Delta (h) = (\mathrm{id}\otimes \Delta)\Delta (h) =
h_{(1)}\otimes h_{(2)}\otimes h_{(3)}$, $h\in H$.
The operation
$F\cdot h = F\Delta^{n-1}(h)$, $F\in H^{\otimes n}$,  $h\in H$, turns
$H^{\otimes n}$ into a right $H$-module (the outer
product of right regular $H$-modules).

A unital left $H$-module $C$ gives rise to an operad (also denoted by $C$) such that
\[
C(n) = \{f:C^{\otimes n} \to H^{\otimes n}\otimes _H C \mid
 f\mbox{ is $H^{\otimes n}$-linear} \}.
\]
For example, if $\dim H=1$ then what we obtain is just a linear space with poly-linear maps.
The composition of such maps as well as the action of a symmetric group is defined in \cite{BDK2001}.

In these terms,
if $\Var $ is a variety of $\Omega $-algebras defined by a system of poly-linear identities $S$ then
a $\Var$-pseudo-algebra structure on an $H$-module $C$ is a functor from $\mathcal P_{\Var} $ to
the operad $C$. Such a functor is determined by a family of $H^{\otimes 2}$-linear maps
$*_i: C\otimes C\to H^{\otimes 2}\otimes _H C$ satisfying the identities (c.f. \cite{Kol2006b})
\[
f ^*(x_1, \dots, x_n)=0, \quad f\in S,\ \deg f=n,
\]
where $f^*$ is obtained from $f$ in the following way.
Assume a poly-linear $\Omega $-monomial $u$ in the variables $x_1, \dots, x_n$
turns into a word $x_{\sigma(1)}\dots x_{\sigma(n)}$, $\sigma \in S_n$,
after removing all brackets and symbols $\circ_i$, $i\in I$. Denote  by
$u^*$ the same monomial $u$ with all $\circ_i$ replaced with $*_i$.
Then $u^*$ can be considered as a map
$C^{\otimes n}\to H^{\otimes n}\otimes_H C$, which is not necessary $H^{\otimes n}$-linear.
However, $u^{(*)}=(\sigma\otimes_H\mathrm{id})u^*$ is $H^{\otimes n}$-linear.
Finally, if $f=\sum\limits_{\xi}\alpha_\xi u_\xi$, $\alpha_\xi\in \Bbbk $,
then
\[
f ^*(x_1, \dots, x_n) = \sum\limits_\xi \alpha_\xi u^{(*)}_\xi.
\]

\begin{exmp}[c.f. \cite{BDK2001}]
Consider an $\Omega$-algebra $A$, a cocommutative bialgebra $H$, and define $C=H\otimes A$.
Then $C$ is an pseudo-algebra with respect to the operations
\[
(f\otimes a)*_i (h\otimes b) = (f\otimes h)\otimes_H (a\circ_i b), \quad f,h\in H,\ a,b\in A,\ i\in I.
\]
Such a pseudo-algebra is denoted by $\mathrm{Cur}\,A$ (current pseudo-algebra).
If $A$ belongs to $\Var$ then, obviously, $\mathrm{Cur}\, A$ is a $\Var$-pseudo-algebra over $H$.
\end{exmp}

Given a pseudo-algebra $C$ with operations $*_i$, $i\in I$, one may define operations
$\vdash_i$, $\dashv_i$ on the same space $C$ as follows:
if $a*_i b = \sum\limits_\xi (h_\xi\otimes f_\xi)\otimes_H d_\xi $
then
\begin{equation}\label{eq:Pseudo2Dialg}
a\vdash_i b  =  \sum\limits_\xi \varepsilon(h_\xi)f_\xi d_\xi,
\quad
a\dashv_i b  =  \sum\limits_\xi h_\xi \varepsilon(f_\xi) d_\xi.
\end{equation}

\begin{prop}\label{prop:Pseudo2Dialg}
Let $C$ be a $\Var $-pseudo-algebra.
Then $C^{(0)}$ is a di-$\Var$-algebra.
\end{prop}

\begin{proof}
It is enough to check that \eqref{eq:0-identities} and \eqref{eq:Dot-identities} hold on $C^{(0)}$.
Indeed, if
$a*_i b = \sum\limits_\xi (h_\xi\otimes f_\xi)\otimes_H d_\xi $,
$d_\xi *_j c = \sum\limits_\eta (h'_\eta\otimes f'_\eta)\otimes_H e_\eta $
then
\[
(a\vdash_i b)*_j c = \sum\limits_\eta \left(\sum\limits_{\xi} \varepsilon(h_\xi)f_\xi d_\xi\right)*_j c
 = \sum\limits_{\eta,\xi} (\varepsilon(h_\xi)f_\xi h'_\eta \otimes f'_\eta)\otimes_H e_\eta.
\]
Hence,
\[
(a\vdash_i b) \vdash_j c = \sum\limits_{\eta,\xi} \varepsilon(h_\xi f_\xi h'_\eta) f'_\eta e_\eta.
\]
 On the other hand,
\[
(a\dashv_i b)*_j c = \sum\limits_\eta \left(\sum\limits_{\xi} h_\xi \varepsilon(f_\xi) d_\xi\right)*_j c
 = \sum\limits_{\eta,\xi} (h_\xi \varepsilon(f_\xi) h'_\eta \otimes f'_\eta)\otimes_H e_\eta,
\]
so
$(a\vdash_i b) \vdash_j c = (a\dashv_i b) \vdash_j c$ for all $a,b,c\in C$.
The second identity in \eqref{eq:0-identities} can be proved in the same way.

Consider a poly-linear identity $f\in S$. This is straightforward to check (c.f. \cite{Kol2008})
that if
\[
 f^{(*)}(a_1, \dots, a_n) = \sum\limits_\xi (h_{1\xi}\otimes \dots \otimes h_{n\xi })\otimes _H c_\xi
\]
then
$f(a_1, \dots, \dot a_k, \dots, a_n) =
 \sum\limits_\xi h_{1\xi} \dots \varepsilon(h_{k\xi}) \dots h_{n\xi } c_\xi$ in $C^{(0)}$.
It is clear that if $f^{(*)}$ vanishes in $C$ then $C^{(0)}$ satisfies the identity
\eqref{eq:Dot-identities}.
\end{proof}

In particular, if $B$ is a $\Var $-algebra then
$(\mathrm{Cur}\, B)^{(0)}$ is a di-$\Var $-algebra.

\begin{prop}\label{prop:Dialg2Current}
If $H$ contains a nonzero element $T$ such that $\varepsilon(T)=0$ then
every di-$\Var $-algebra $A$ embeds into
$(\mathrm{Cur}\,\hat A)^{(0)}$.
\end{prop}

\begin{proof}
Recall that $\hat A = \bar A\oplus A$,
$\mathrm{Cur}\, \hat A = H\otimes \hat A$.
Define
\begin{equation}\label{eq:CurrEmbedding}
\iota : A \to  H\otimes \hat A,
\quad
\iota(a) = 1\otimes \bar a + T\otimes a.
\end{equation}
This map is obviously
injective, and
\[
\iota(a)*_i \iota(b)
    = (1\otimes 1)\otimes_H (1\otimes \overline{a\vdash_i b})
 + (T\otimes 1)\otimes _H (1\otimes a\dashv_i b) + (1\otimes T)\otimes _H (1\otimes a\vdash_i b).
\]
Since $\overline{a\vdash_i b}= \overline{a\dashv_i b}$ in $\hat A$,
we have
\[
\begin{gathered}
 \iota(a)\vdash_i \iota(b) =1\otimes \overline{a\vdash_i b} + T\otimes a\vdash_i b= \iota(a\vdash_i b), \\
 \iota(a)\dashv_i \iota(b) =1\otimes \overline{a\dashv_i b} + T\otimes a\dashv_i b= \iota(a\dashv_i b).
\end{gathered}
\]
\end{proof}

\section{Dendriform di- and trialgebras}

The operad $\Dend$ of associative dendriform algebras is known to be
Koszul dual (see \cite{GinzKapr} for details of Koszul duality)
to the operad $\Dias=\mathcal D_{\mathrm{As}}$
of diassociative algebras.
Since $\Dias \simeq \mathrm{As}\otimes \Perm$
and it was noticed in \cite{Vallette2008} that for $\Perm $
(as well as for $\CommTrias $) the Hadamard product $\otimes $
coincides with the Manin white product $\circ$,
 we have
$\Dend : = (\mathrm{As}\otimes \Perm)^! = \mathrm{As} \bullet \PreLie$,
where $\mathrm{As}^! = \mathrm{As}$, $\PreLie $ is the operad of
pre-Lie algebras which is Koszul dual to $\Perm $,
$\bullet $ stands for the Manin black product of operads
\cite{GinzKapr}.

In general, for a binary operad $\mathcal P$
the {\em successor procedure\/} described in \cite{BBGN2011}
gives rise to what is natural to call
defining identities of di- or tri-$\mathcal P$-dendriform algebras.
In addition, if $\mathcal P$ is quadratic then these $\mathcal P$-dendriform
algebras are dual to the corresponding di- or tri-$\mathcal P^!$-algebras.
In this case, obviously,  $(\mathcal P^!\otimes \Perm)^! = P\bullet \PreLie $
for dialgebras, and
$(\mathcal P^!\otimes \CommTrias)^! = \mathcal P^!\bullet \PostLie$,
where
$\PreLie = \Perm^!$
$\PostLie = \CommTrias^!$.
This is closely related with Proposition \ref{prop:KoszulDual} below.

In terms of identities, we do not need $\mathcal P$ to be quadratic
(in fact, it is easy to generalize the successor procedure even
for algebras with $n$-ary operations,  $n\ge 2$).

Usually, the operations in a dendriform di- or trialgebra are denoted
by $\prec$, $\succ $, and $\cdot $. We will use $\dashv$, $\vdash$,
and $\perp $ instead.

Suppose $\Var $ is a variety of $\Omega $-algebras defined by
a family $S$ of poly-linear identities, as above.

\begin{defn}
A {\em tri-$\Var$-dendriform algebra} is an $\Omega^{(3)}$-algebra satisfying
the identities
\begin{equation}\label{eq:TriDendId}
f^*(x_1,\dots , \dot x_{k_1}, \dots , \dot x_{k_l}, \dots, x_n),
\quad f\in S, \ n=\deg f, \ 1\le k_1<\dots <k_l\le n,
\end{equation}
for all $l=1,\dots, n$,
where
$f^*(x_1,\dots , \dot x_{k_1}, \dots , \dot x_{k_l}, \dots, x_n)$
is obtained from $f$ by means of the following procedure.
\end{defn}

If $u$ is a non-associative monomial, $u=vw$, $H=\{k_1,\dots, k_l\}$,
$\deg v=p$,
$H=H_1\dot\cup H_2$
as above, then
\[
\Phi^*(n)(u^H) = \begin{cases}
 \Phi^*(p)(v^{H_1}) \perp_i  \Phi^*(n-p)(w^{H_2}), & \mbox{if } H_1,H_2\ne \emptyset, \\
 v^{*}\vdash_i \Phi^*(n-p)(w^H), & \mbox{if } H_1=\emptyset, \\
\Phi^*(p)(v^H)\dashv_i w^{*}, & \mbox{if } H_2=\emptyset,
 \end{cases}
\]
where $v^*$  stands for the linear combination of trees obtained when we replace
each label $\circ_j$ in $v$ with $\vdash_j+\dashv_j +\perp_j$).

Now, extend $\Phi^*(n)$ by linearity and set

\[
f^*(x_1,\dots , \dot x_{k_1}, \dots , \dot x_{k_l}, \dots, x_n):=
\sum\limits_{\xi }\alpha_\xi  \Phi^*(n)(u^H_\xi)
\]
for
$f(x_1,\dots, x_n) = \sum\limits_{\xi }\alpha_\xi u_\xi \in \mathcal F(n)$,
$\alpha_\xi \in \Bbbk $,
$H=\{k_1,\dots, k_l\}$.

To get the definition of a di-$\Var $-dendriform algebra, it is enough to
set $x\perp y = 0$ and consider $|H|=1$ only.

Denote by $\dend\mathcal {D}_{\Var }$ and $\dend\mathcal T_{\Var }$
the operads governing di- and tri-$\Var $-dend\-ri\-form algebras.

\begin{prop}\label{prop:KoszulDual}
Suppose $|\Omega |<\infty $.
If $\mathcal P_{\Var} $ is a binary quadratic operad 
with $\mathcal P_{\Var} (1)=\Bbbk $ then
$(\mathcal D_{\Var})^! = \dend\mathcal D_{\Var^!}$ and
$(\mathcal T_{\Var})^! = \dend\mathcal T_{\Var^!}$,
where $\Var ^!$ stands for the class of algebras governed
by the Koszul dual operad $\mathcal P_{\Var }^!$.
\end{prop}

\begin{proof}
 We consider trialgebra case in details since it covers the dialgebra case.

Suppose $\mathcal P_{\Var} = \mathcal P(E,R)$ is a binary quadratic operad,
i.e., a quotient operad of $\mathcal F$, $\mathcal F(2)=E$, with respect to the operad
ideal generated by $S_3$-submodule $R\subseteq \mathcal F(3)$,
see \cite{GinzKapr} for details.

The space $E$ is spanned by $\mu_i: x_1\otimes x_2 \mapsto x_1\circ_i x_2$
and $\mu_i^{(12)}: x_1\otimes x_2 \mapsto x_2\circ_i x_1$, $i\in I$.
Without loss of generality, we may assume that
$\mu_i$, $i\in I$, are linearly independent and
\[
 \mu_k^{(12)} = \sum\limits_{i\in I} \alpha _{ik} \mu_i,
  \quad  k\in I'\subseteq I,\ \alpha_{ik}\in \Bbbk ,
\]
are the only defining identities of $\Var $ of degree two, $|I'|=d\ge 0$
(if $\mathrm{char}\,\Bbbk \ne 2$, these are just commutativity and
anti-commutativity).
Denote by $N=2|I|-d$ the dimension of $E$.

The space $\mathcal F(3)$ can be naturally identified with the induced $S_3$-module
$\Bbbk S_3 \otimes _{\Bbbk S_2} (E\otimes E)$, where $E\otimes E$ is considered as an
$S_2$-module via $(\mu\otimes \nu)^{(12)} = \mu\otimes \nu^{(12)}$, $\mu, \nu\in E$.
Namely, the basis of $\mathcal F(3)$ consists of expressions
\[
\sigma \otimes_{\Bbbk S_2} (\mu\otimes \nu),\quad \sigma \in \{e,(13),(23) \},
 \]
$\mu$ and $\nu$ range over a chosen basis of $E$. Therefore, $\dim \mathcal F(3) = 3N^2$.

In terms of monomials (or binary trees), for example,
$e\otimes_{\Bbbk S_2}(\mu_i\otimes \mu_j)$ corresponds to
$(x_1\circ_j x_2)\circ_i x_3$,
$e\otimes_{\Bbbk S_2}(\mu_i^{(12)}\otimes \mu_j)$ to
$x_3\circ_i (x_1\circ_j x_2)$. A permutation $\sigma \in S_3$ in the first tensor
factor permutes variables, e.g.,
$(13)\otimes_{\Bbbk S_2} (\mu_i^{(12)}\otimes \mu_j^{(12)})$
corresponds to $x_1\circ_i (x_2\circ_j x_3)$.

Recall that $E^\vee $ denotes the dual space to $E$ considered as
an $S_2$-module with respect to sgn-twisted action:
$\langle \nu^{(12)}, \mu\rangle = -\langle \nu, \mu^{(12)}\rangle $,
$\nu \in E^\vee $, $\mu \in E$.
If $\mathcal F^\vee $ is the free binary operad generated by $E^\vee $
then
$(\mathcal F(3))^\vee \simeq \mathcal F^\vee (3)
=\Bbbk S_3 \otimes_{\Bbbk S_2}(E^\vee\otimes E^\vee) $.

The Koszul-dual operad $\mathcal P_{\Var }^!$
is the quotient of $\mathcal F^\vee$ by the operad ideal
generated by $R^\perp \subset \mathcal F^\vee(3)$, the orthogonal space to~$R$.

By the definition, the operad $\mathcal T_\Var $
governing the variety of tri-$\Var $-algebras
is equal to $\mathcal P(E^{(3)}, R^{(3)})$,
where the initial data $E^{(3)}$, $R^{(3)}$ are defined as follows.
The space $E^{(3)}$ is spanned by
$\mu_i^{*}$, $(\mu_i^{*})^{(12)}$, $i\in I$, $*\in \{\vdash, \dashv, \perp\}$,
with respect to the relations
\[
\begin{gathered}
(\mu_k^{\vdash})^{(12)} = \sum\limits_{i\in I}\alpha _{ik} \mu^{\dashv}_i,
\quad
(\mu_k^{\dashv})^{(12)} = \sum\limits_{i\in I}\alpha _{ik} \mu^{\vdash}_i \\
(\mu_k^{\perp})^{(12)} = \sum\limits_{i\in I}\alpha _{ik} \mu^{\perp}_i,
\quad  k\in I'.
\end{gathered}
\]
The $S_3$-module $R^{(3)}$ is generated by the defining identities of
tri-$\Var$-algebras, i.e.,
\[
R^{(3)} = \{\Phi(3)(f^H)\mid f\in R,\, \emptyset\ne H\subseteq \{1,2,3\} \} \oplus O^{(3)},
\]
and
$O^{(3)}$ is the $S_3$-submodule of $\mathcal F^{(3)}$ generated by
\begin{equation}\label{eq:0Ident_mu}
\begin{gathered}
 \mu_j^\vdash \otimes \mu_i^{\dashv} - \mu_j^{\vdash}\otimes \mu_i^{\vdash}, \quad
 \mu_j^\vdash \otimes \mu_i^{\perp} - \mu_j^{\vdash}\otimes \mu_i^{\vdash}, \\
 (\mu_i^\dashv)^{(12)} \otimes \mu_j^{\vdash} - (\mu_i^{\dashv})^{(12)}\otimes \mu_j^{\dashv}, \quad
 (\mu_i^\dashv)^{(12)} \otimes \mu_j^{\perp} - (\mu_i^{\dashv})^{(12)}\otimes \mu_j^{\perp},\\
i,j\in I.
\end{gathered}
\end{equation}
This is easy to calculate that
$\dim E^{(3)}=3N$, $\dim \mathcal F^{(3)}(3)=27N^2$,
$\dim O^{(3)} = 6N^2$, so $\dim R^{(3)} = 6N^2+7\dim R$.
Denote by
$O^{(3)}_+$ the $S_3$-submodule of $\mathcal F^{(3)}$ generated by
the first summands of all relations from \eqref{eq:0Ident_mu}.

Suppose $f\in \mathcal F(3)$, $g\in \mathcal F^\vee (3)$,
and let $H_1,H_2\subseteq \{1,2,3\}$ be nonempty subsets.
It follows from the definition of $\Phi (3)$ that
 $\langle \Phi(3)(f^{H_1}) , \Phi(3)(g^{H_2}) \rangle =0$ if $H_1\ne H_2$.
For $H_1=H_2=H$, orthogonality of $f$ and $g$ implies
$\langle \Phi(3)(f^{H}) , \Phi(3)(g^{H}) \rangle =0$ as well.
Moreover, for every $f\in \mathcal F(3)$ we have
$\langle \Phi(3)(f^H), O^{(3)}_+ \rangle =0$ since neither of terms
from $O^{(3)}_+$ appears in images of $\Phi(3)$.

This is now easy to see that if $g\in R^\perp \subseteq \mathcal F^{\vee}(3)$
then $\langle f, \Phi^*(3)(g^{H}) \rangle = 0$ for every $f\in  R^{(3)}$.
Hence,
\[
 (R^\perp)^{(3*)}:=\{ \Phi^*(3)(g^{H}) \mid g\in R^\perp,\, \emptyset\ne H\subseteq \{1,2,3\}\}\subseteq
 (R^{(3)})^\perp.
\]
On the other hand, $\dim R^{\perp} = 3N^2-\dim R$, so $\dim (R^\perp)^{(3*)} = 21N^2-7\dim R$.
Therefore,
$\dim (R^\perp)^{(3*)} + \dim R^{(3)} = 27N^2 $ and $ (R^\perp)^{(3*)}=(R^{(3)})^\perp$.
It remains to recall that, by definition,
$\mathrm{dend}\mathcal T_{\Var} = \mathcal P(E^{(3)}, (R^\perp)^{(3*)})$.
\end{proof}

\section{Embedding into Rota---Baxter algebras}

Suppose $B$ is an $\Omega $-algebra.
A linear map $R:B\to B$ is called a
Rota---Baxter operator of weight $\lambda \in \Bbbk $  if
\begin{equation}\label{eq:RB-Operator}
R(x)\circ_i R(y) = R(x\circ_i R(y) + R(x)\circ_i y + \lambda x\circ_i y)
\end{equation}
for all $x,y\in B$, $i\in I$.

Let $A$ be an $\Omega^{(3)}$-algebra.
Consider the isomorphic copy $A'$ of the underlying linear space
$A$ (assume $a\in A$ is in the one-to-one correspondence with $a'\in A'$),
and define the following $\Omega $-algebra structure
on the space $\hat A = A\oplus A'$:
\begin{equation}\label{eq:Hat-Defn}
\begin{gathered}
a\circ_i b = a\vdash _i b + a\dashv_i b + a\perp_i b,
\quad
a\circ_i b' = (a\vdash _i b)',
\\
a'\circ_i b = (a\dashv_i b)',
\quad
a'\circ_i b' = (a\perp_i b)',
\end{gathered}
\end{equation}
for $a,b\in A$, $i\in I$.

\begin{lem}\label{lem:RB-hat}
Given a scalar $\lambda \in \Bbbk $, the linear map $R:\hat A \to \hat A$
defined by $R(a')=\lambda a$, $R(a)=-\lambda a$ ($a\in A$)
is a Rota---Baxter operator of weight $\lambda $
on the $\Omega $-algebra $\hat A$.
\end{lem}

\begin{proof}
It is enough to check the relation \eqref{eq:RB-Operator}.
Straightforward computation shows
\begin{multline}\nonumber
R(a+b')\circ_i R(x+y')
=\lambda^2(-a+b)\circ_i (-x+y) \\
= \lambda^2 (
 a\vdash_i x + a\dashv_i x + a\perp_i x
-a\vdash_i y - a\dashv_i y - a\perp_i y \\
-b\vdash_i x - b\dashv_i x - b\perp_i x
+b\vdash_i y + b\dashv_i y + b\perp_i y).
\end{multline}
On the other hand,
\begin{multline}\nonumber
R((a+b')\circ_i R(x+y') + R(a+b')\circ_i (x+y') + \lambda (a+b')\circ (x+y')) \\
=
\lambda R((a+b')\circ_i (-x+y) + (-a+b)\circ_i (x+y') + (a+b')\circ (x+y')) \\
=
\lambda R(
-a\vdash_i x - a\dashv_i x - a\perp_i x
+a\vdash_i y + a\dashv_i y + a\perp_i y \\
- (b\dashv_i x)' + (b\dashv_i y)'
-a\vdash_i x - a\dashv_i x - a\perp_i x
+ b\vdash_i x + b\dashv_i x + b\perp_i x \\
- (a\vdash_i y)' + (b\vdash_i y)'
+a\vdash_i x + a\dashv_i x + a\perp_i x
+ (a\vdash_i y)' + (b\dashv_i x)'
+ (b\perp_i y)'
)    \\
=
\lambda^2
(
-a\vdash_i y - a\dashv_i y - a\perp_i y
 + b\dashv_i y
+a\vdash_i x + a\dashv_i x \\
+ a\perp_i x
- b\vdash_i x - b\dashv_i x - b\perp_i x
 + b\vdash_i y
+ b\perp_i y
).
\end{multline}
\end{proof}

\begin{lem}\label{lem:RB-dialg}
Let $A$ be a di-$\Var$-dendriform algebra. Then the
map $R:\hat A\to \hat A$ defined by
$R(a')=a$, $R(a)=0$ is a Rota---Baxter operator of weight $\lambda =0$
on $\hat A$.
\end{lem}

The proof is completely analogous to the previous one.

The following statement is well-known (c.f. \cite{Aguiar00, Fard02, Uchino}),
but we will state its proof for readers' convenience.

\begin{prop}\label{prop:RB_to_dendriform}
Let $B$ be an $\Omega $-algebra with a Rota---Baxter
operator $R$ of weight $\lambda \ne 0$. Assume $B$
belongs to $\Var $.
Then the same linear space $B$ considered as
$\Omega^{(3)}$-algebra with respect to the operations
\begin{equation}\label{eq:RB-DendOperations}
x\vdash_i y = \frac{1}{\lambda }R(x)\circ_i y, \quad
x\dashv_i y = \frac{1}{\lambda } x\circ_i R(y), \quad
x\perp_i y = x\circ_i y.
\end{equation}
is a tri-$\Var$-dendriform algebra.
\end{prop}

\begin{proof}
Let $u=u(x_1,\dots, x_n) \in \mathcal F(n)$ be a poly-linear
$\Omega $-monomial.
The claim follows from the following relation in $B$:
\begin{equation}\label{eq:RB-monomial}
u^*(x_1,\dots, \dot x_{k_1}, \dots, \dot x_{k_l}, \dots , x_n)
=
\frac{1}{\lambda^{n-l}}u(R(x_1), \dots, x_{k_1}, \dots, x_{k_l}, \dots, R(x_n)),
\end{equation}
i.e., in order to get a value of an $\Omega^{(3)}$-monomial in $\hat A$
we have to replace every non-emphasized variable $x_i$
($i\notin H=\{k_1,\dots, k_l \}$) with $\frac{1}{\lambda}R(x_i)$.

Relation \eqref{eq:RB-monomial} is clear for $n=1,2$.
In order to apply induction on $n$, we have to start with the case when
$H=\emptyset $. Recall that $u^*(x_1,\dots, x_n)$
stands for the expression obtained from $u$ by means of replacing
each $\circ_i$ with $\vdash_i + \dashv_i + \perp_i$.
Then
\begin{equation}\label{eq:RB-star-monomial}
R(u^*(x_1,\dots, x_n)) = \frac{1}{\lambda^{n-1}}u(R(x_1), \dots, R(x_n)),
\quad n\ge 2,
\end{equation}
in $\hat A$. Indeed, for $n=2$ we have exactly the Rota---Baxter relation.
If $u = v\circ_i w$, $v=v(x_1,\dots, x_p)$, $w=w(x_{p+1}, \dots, x_n)$, then, by induction,
\begin{multline}\nonumber
R(u^*)= R(v^* \vdash_i w^* + v^* \dashv_i w^* + v^* \perp_i w^*)\\
= \frac{1}{\lambda }R(R(v^*)\circ_i w^* + v^* \circ_i R(w^*)
           +  \lambda v^* \perp_i w^*)
=\frac{1}{\lambda } R(v^*)\circ_i R(w^*)  \\
= \frac{1}{\lambda } \frac{1}{\lambda^{p-1} }v(R(x_1), \dots, R(x_p))\circ_i
  \frac{1}{\lambda^{n-p-1} } w(R(x_{p+1}), \dots, R(x_n)) \\
= \frac{1}{\lambda^{n-1} } u(R(x_1), \dots, R(x_n)).
\end{multline}

Now, let us finish proving \eqref{eq:RB-monomial}.
If $u = v\circ_i w$, $\deg v =p$, $H=H_1\dot\cup H_2$
then there are three cases: (a) $H_1,H_2\ne \emptyset $;
(b) $H_1=\emptyset $; (c) $H_2=\emptyset $.

In the case (a),
$u^*(x_1,\dots , \dot x_{k_1}, \dots , \dot x_{k_l}, \dots, x_n)
 = \Phi^*(n) (u^H) = \Phi^*(p)(v^{H_1}) \perp_i \Phi^*(n) (w^{H_2})$,
and it remains to apply inductive assumption and the definition of
$\perp_i$ from \eqref{eq:RB-DendOperations}.

In the case (b),
$ \Phi^*(n) (u^H) = v^*\vdash_i \Phi^*(n-p) (w^{H})$,
so for any $a_1, \dots, a_n \in B$ we
can apply
\eqref{eq:RB-star-monomial}
to get
\begin{multline}\nonumber
 [\Phi^*(n) (u^H)] (a_1,\dots, a_n)=
\frac{1}{\lambda } R(v^*(a_1,\dots, a_p))
 \circ_i [\Phi^*(n-p) (w^{H})](a_{p+1}, \dots, a_n) \\
 =
\frac{1}{\lambda ^{p}} v(R(a_1), \dots, R(a_p))\circ_i
\frac{1}{\lambda ^{n-p-l}} w(R(a_{p+1}, \dots , a_{k_1}, \dots, a_{k_l},
\dots , R(a_{n})) \\
=
\frac{1}{\lambda^{n-l}}u(R(a_{1}, \dots , a_{k_1}, \dots, a_{k_l}, \dots , R(a_{n})).
\end{multline}

The case (c) is completely analogous.
\end{proof}

\begin{remark}
Following \cite{Fard02}, the dendriform operations on a
Rota---Baxter algebra should be defined as
\begin{equation}\label{eq:RB-DendAF}
x\vdash_i y = R(x)\circ_i y,\quad
x\dashv_i y = x\circ_iR(y),\quad
x\perp_i y = \lambda x\circ_i y.
\end{equation}
It is easy to see that for $\lambda \ne 0$
one should just re-scale these binary operations to get
\eqref{eq:RB-DendOperations}.
\end{remark}

\begin{prop}[\cite{Aguiar00, Uchino}]\label{prop:RB_to_dendriform2}
Let $B$ be an $\Omega $-algebra with a Rota---Baxter
operator $R$ of weight $\lambda = 0$. Assume $B$
belongs to $\Var $.
Then the same linear space $B$ considered as
$\Omega^{(2)}$-algebra with respect to
$x\vdash_i y = R(x)\circ_i y$,
$x\dashv_i y = x\circ_iR(y)$
is a di-$\Var $-dendriform algebra.
\end{prop}

The proof is analogous to the proof of the previous statement.

\begin{thm}\label{thm:DoubleEmbedding}
The following statements are equivalent:
\begin{enumerate}
\item $A$ is a tri-$\Var $-dendriform algebra;
\item $\hat A$ belongs to $\Var $.
\end{enumerate}
\end{thm}

\begin{proof}
(1) Assume $A$ is a tri-$\Var $-dendriform algebra,
and let $S$ be the set of defining identities of $\Var $.
We have to check that every $f\in S$ holds on $\hat A$.

First, let us compute a monomial in $\hat A=A\oplus A'$ when all
its arguments belong to the first summand.

\begin{lem}\label{lem:NoEmphasized}
Suppose $u=u(x_1,\dots, x_n)\in \mathcal F(n)$
is a poly-linear $\Omega$-monomial of degree~$n$.
Then in the $\Omega $-algebra $\hat A$ we have
\begin{equation}\label{eq:NoEmphasized}
u(a_1,\dots, a_n) =
 \sum\limits_{H} \Phi^*(n)(u^H)(a_1,\dots, a_n), \quad a_i\in A,
\end{equation}
where $H$ ranges over all nonempty subsets of $\{1,\dots, n\}$.
\end{lem}

\begin{proof}
By the definition of multiplication in $\hat A$,
$u(a_1,\dots, a_n) = u^*(a_1,\dots, a_n)$,
where $u^*$ means the same as in the definition of $\Phi^*(n)$.
In particular, for $n=1,2$ the statement is clear.
 Proceed by induction on $n=\deg u$.
Assume
$u=v\circ _i w$, and, without loss of generality,
$v=v(x_1,\dots, x_p)$, $w=w(x_{p+1}, \dots, x_n)$.
Then
\begin{multline}\label{eq:MonomialExpand}
u(a_1,\dots, a_n)
=
v^*(a_1,\dots, a_p)\vdash_i
 \left (\sum\limits_{H_2} \Phi^*(n-p)(w^{H_2})(a_{p+1},\dots, a_n) \right ) \\
+
\left (\sum\limits_{H_1} \Phi^*(p)(v^{H_1})(a_1,\dots, a_p) \right )
\perp_i
\left (\sum\limits_{H_2} \Phi^*(n-p)(w^{H_2})(a_{p+1},\dots, a_n) \right ) \\
+
\left (\sum\limits_{H_1} \Phi^*(p)(v^{H_1})(a_1,\dots, a_p) \right )
 \dashv_i w^*(a_{p+1}, \dots, a_n),
\end{multline}
where $H_1$ and $H_2$ range over all nonempty subsets of
$\{1,\dots, p\}$ and $\{p+1, \dots, n\}$, respectively.
It is easy to see that the overall sum is exactly the right-hand side
of \eqref{eq:NoEmphasized}: The first (second, third) group of summands in
\eqref{eq:MonomialExpand} corresponds to $H=H_2\subseteq \{p+1,\dots, n\}$,
($H=H_1\cup H_2$, $H=H_1\subseteq \{1,\dots, p\}$, respectively).
\end{proof}

Next, assume that $l>0$ arguments belong to $A'$.

\begin{lem}\label{lem:SomeEmphasized}
Suppose $u=u(x_1,\dots, x_n)\in \mathcal F(n)$
is a poly-linear $\Omega$-monomial of degree~$n$,
$H=\{k_1,\dots, k_l\}$ is a nonempty subset of
$\{1, \dots, n\}$.
Then in the $\Omega $-algebra $\hat A$ we have
\begin{equation}
u(a_1,\dots, a_{k_1}', \dots, a_{k_l}', \dots, a_n)
=
\big( \Phi^*(n)(u^H)(a_1, \dots, a_n) \big )'.
\end{equation}
\end{lem}

\begin{proof}
For $n=1,2$ the statement is clear.
If $u=v\circ_i w$ for some $i\in I$ as above then we have to consider three
natural cases: (a) $H\subseteq \{1,\dots, p\}$; (b) $H\subseteq
 \{p+1, \dots, n\}$; (c) variables with indices from $H$
appear in both $v$ and $w$.

In the case (a), the inductive assumption implies
\begin{multline}\nonumber
u(a_1,\dots, a_{k_1}', \dots, a_{k_l}', \dots, a_n)   \\
 =
v(a_1,\dots, a_{k_1}', \dots, a_{k_l}', \dots, a_p)\dashv_i
w^*(a_{p+1}, \dots, a_n) \\
=
\big(
 \Phi^*(p)(v^H)(a_1,\dots, a_p) \dashv_i w^*(a_{p+1}, \dots, a_n)\big)',
\end{multline}
and it remains to recall the definition
of $\Phi^*(n)$. Case (b) is analogous.

In the case (c), $H=H_1\dot\cup H_2$ as above and
\begin{multline}\nonumber
u(a_1,\dots, a_{k_1}', \dots, a_{k_l}', \dots, a_n) \\
=
 \Phi^*(p)(v^{H_1})(a_1,\dots, a_p)\perp_i
 \Phi^*(n-p)(w^{H_2})(a_{p+1},\dots, a_n)
\end{multline}
that proves the claim.
\end{proof}

Finally, suppose $f\in S$ is a poly-linear identity of degree $n$.
Then $\Phi^*(n)(f^H)$ is an identity on the $\Omega^{(3)}$-algebra
$A$, so Lemmas \ref{lem:NoEmphasized} and \ref{lem:SomeEmphasized}
imply $f$ to hold on $\hat A$.

(2) The map $\iota: A\to \hat A$, $\iota(a)=a'$, is an embedding of
the $\Omega^{(3)}$-algebra $A$ into $\hat A$ equipped with operations
\eqref{eq:RB-DendOperations}. By Proposition~\ref{prop:RB_to_dendriform},
$\hat A$ is a tri-$\Var $-dendriform algebra, therefore, so is~$A$.
\end{proof}

\begin{remark}
Since $\lambda \ne 0$, we may conclude that every tri-$\Var$-dendriform algebra
$A$ embeds into a Rota---Baxter algebra $B\in \Var$ of weight $\lambda $
in the sense of Aguiar \cite{Aguiar00} (see \eqref{eq:RB-DendAF}):
It is sufficient to re-scale the product on $B$.
\end{remark}

If $\lambda =0 $ then the simple reduction of
Theorem~\ref{thm:DoubleEmbedding} by means of
Lemma~\ref{lem:RB-dialg} leads to

\begin{thm}\label{thm:DoubleEmbedding-2}
Suppose $A$ is an $\Omega^{(2)}$-algebra, and let
$\hat A$ stands for an $\Omega $-algebra defined by
\eqref{eq:Hat-Defn} with $x\perp_i y\equiv 0$.
Then  the following statements are equivalent:
\begin{enumerate}
\item $A$ is a di-$\Var $-dendriform algebra;
\item $\hat A$ belongs to $\Var $.
\end{enumerate}
\end{thm}

\begin{remark}
It is interesting to note that $A$ is a simple
di-$\Var $-dendriform algebra if and only if
$\hat A$ is a simple Rota---Baxter algebra.
\end{remark}

The standard reasoning
allows to conclude the following.

\begin{cor}[c.f. \cite{Chen11}]\label{cor:Chen}
Every di-$\Var $-dendriform algebra embeds into its
universal enveloping Rota---Baxter $\Var $-algebra of
weight $\lambda = 0$.
\end{cor}

\begin{cor}\label{cor:Tri-Embedding}
Every tri-$\Var $-dendriform algebra embeds into its
universal enveloping Rota---Baxter $\Var $-algebra of
weight $\lambda \ne 0$.
\end{cor}

\begin{remark}
 All results of this section remain valid for algebras over a commutative ring $K$
if we replace the condition $\lambda \ne 0$ with $\lambda \in K^*$, 
where $K^*$ is the set of invertible elements of $K$.
\end{remark}

\section{Skew trialgebras and Rota---Baxter algebras}

Consider a slightly modified analogue of trialgebras which
we shortly call s-trialgebras.

\begin{defn}
A {\em s-tri-$\Var $-algebra\/} is an
$\Omega^{(3)}$-algebra satisfying the identities
\eqref{eq:0-identities}, \eqref{eq:triDot-Identities}.
\end{defn}

In other words, we exclude the identities
$x_1\dashv_i (x_2 \perp_j x_3) = x_1\dashv_i (x_2 \dashv_j x_3) $,
$(x_1\perp_i x_2) \vdash_j x_3 = (x_1\vdash_i x_2) \vdash_j x_3 $
from the definition of a tri-$\Var$-algebra.

For any $\Omega^{(3)}$-algebra $A$ satisfying the identities \eqref{eq:0-identities}
we can also construct (as in the dialgebra case) the $\Omega$-algebra
$\hat{A}=\bar{A}\oplus A$ as follows (similar to \eqref{eq:EilenbergDialg}):
$\bar{A}=
A/\mathrm{Span}\,\{a\vdash_i b-a\dashv_i b\mid a,b\in A,\, i\in I\}$,
$\bar a\circ_i \bar b=\overline{a\vdash_i b}$,
$\bar a\circ_i b = a\vdash_i b$, $a\circ_i \bar b = a\dashv_i b$,
$a\circ_i b = a\perp_i b$.
An analogue of Proposition \ref{prop:EilenbergDouble} holds for this construction,
i.e., it gives an equivalent definition of a s-tri-$\Var$-algebra.

It turns out that s-tri-$\Var $-algebras are closely related with
$\Gamma $-conformal algebras  introduced in \cite{Kac&Gol}.
These systems appeared as  ''discrete analogues'' of conformal algebras defined over a group $\Gamma$.
From the general point of view, these are pseudo-algebras over the group algebra $H=\Bbbk \Gamma $
considered as a Hopf algebra with respect to canonical coproduct $\Delta (\gamma )=\gamma \otimes \gamma $
 and counit $\varepsilon(\gamma )=1$, $\gamma \in \Gamma $.

We consider the case when $\Gamma =\langle\mathbb{Z},+\rangle$, $H=\Bbbk[t,t^{-1}]$.
If $C$ is a pseudo-algebra over $H$, i.e., a $\mathbb Z$-conformal algebra with operations
$*_i$, $i\in I$, then
for every $a,b\in C$ their pseudo-product $a*_i b\in H^{\otimes 2}\otimes_H C$ can be
presented as
\[
a*_i b = \sum\limits_{n\in \mathbb Z} (t^{-n}\otimes 1) \otimes_H c_n,
\]
where almost all $c_n$ are zero. It is convenient to denote $c_n $ by $a_{(n)} b$ \cite{Kac&Gol}.
These operations provide an equivalent definition of a $\mathbb Z$-conformal algebra:
This is a linear space with bilinear operations $\{ (\cdot_{(n)} \cdot) \mid n \in \mathbb{Z}\}$
and with a linear invertible mapping $t$ such that the following axioms are satisfied:

\begin{itemize}

\item[(Z1)] $a_{(n)}b = 0$ for almost all $n\in\mathbb{Z}$;

\item[(Z2)] $ta _{(n)}b = a_{(n+1)}b$;

\item[(Z3)] $t(a_{(n)}b)=ta_{(n)}tb$.

\end{itemize}

A $\mathbb Z$-conformal algebra $C$ is associative if
$a_{(n)}(b_{(m)}c)=(a_{(n-m)}b)_{(m)}c$
for all $n,m\in \mathbb Z$, $a,b,c\in C$.

Proposition \ref{prop:Pseudo2Dialg} implies that every di-$\Var $-algebra
can be embedded into a current $\mathbb Z$-conformal algebra over
an algebra from $\Var $ (one may consider, e.g., $T=1-t$).
For s-trialgebras, a similar statement holds.

\begin{exmp}
Let $C$ be an associative $\mathbb{Z}$-conformal algebra. Then,
with respect to the operations $\dashv$, $\vdash$ from \eqref{eq:Pseudo2Dialg}
and $a\perp b=a_{(0)}b$, the vector space $C$ is a s-tri-associative algebra.
Let us denote it also by $C^{(0)}$.
\end{exmp}

There is an interesting question: Whether a trialgebra or s-trialgebra $A$
can be embedded into $C^{(0)}$ for some $\mathbb{Z}$-conformal algebra $C$.
We have a positive answer for Loday trialgebras and only for $\mathrm{char}\,\Bbbk =2$.
Then the mapping $\phi$ from \eqref{eq:CurrEmbedding} realizes this embedding
of $A$ into $\mathrm{Cur}\,\hat A$.

\begin{exmp}
A vector space $A$ endowed with two binary operations $\vdash$, $\perp$
belongs to the variety $\sCommTrias$
(skew commutative tri-associative algebras) if both operations are
associative, $\perp$ is commutative and they also satisfy the
following identities:
\[
x_1\vdash (x_2\perp x_3)=(x_1\vdash x_2)\perp x_3, \quad
(x_1\vdash x_2)\vdash x_3=(x_2\vdash x_1)\vdash x_3.
\]

This is easy to derive from the definition that
free $\sCommTrias[X]$ algebra is nothing but $\Perm\langle\Comm[X]\rangle$,
its linear basis consists of words
$$
u_{1}\vdash u_{1}\vdash\ldots \vdash u_{k}\vdash u_{0}, \quad
u_{1}\leqslant\ldots\leqslant u_{k},
$$
where $u_i$ are basic monomials of the polynomial algebra
$\Comm[X]$ with respect to the operation
$\perp$ and some linear ordering $\leqslant $.
\end{exmp}

\begin{exmp}
Let $\langle A,\cdot\rangle$ be an associative algebra with a derivation $d$
such that $d^2=0$. Defining $a\vdash b=d(a)b$, $a\dashv b=ad(b)$ we obtain
s-tri-associative algebra $(A,\vdash,\dashv,\cdot)$.
\end{exmp}

\begin{exmp}
An associative s-trialgebra $A$ with respect to the
operations $[x,y]=x\dashv y-x\vdash y$ and $x\cdot y=x\perp y$ turns into a dialgebra
analogue of a Poisson algebra: The operation $[\cdot, \cdot]$ satisfies the
Leibniz identity and $\cdot $ is associative.
Moreover, the Poisson identity holds:
\[
[xy,z]=x[y,z]+[x,z]y.
\]
In \cite{Trialg01}, the same operations $[\cdot, \cdot]$ and $\cdot$ were considered for ordinary
triassociative algebra (in the sense of Definition \ref{defn:TriVarAlgebra}).
The noncommutative analogue of a Poisson algebra obtained in this way
 satisfies one more identity $[x,yz-zy]=[x,[y,z]]$ which does not appear in the case of
s-tri-associative algebras.
\end{exmp}

Let us define a class of ''skew'' dendriform algebras associated with a variety $\Var $ of algebras.

\begin{defn}
A {\em s-tri-$\Var$-dendriform algebra} is an $\Omega^{(3)}$-algebra satisfying
\[
(x_1 \perp_i x_2) \vdash_j x_3 = 0, \quad x_1\dashv_i (x_2 \perp_j x_3) = 0, \quad i,j\in I,
\]
and the analogues of identities \eqref{eq:TriDendId} with the following difference:
to define $v^*$ one should replace
$\circ_j $ with $\dashv_j+\vdash_j$.
\end{defn}

As above, the class of s-tri-$\Var$-dendriform algebras
is Koszul dual to the class of s-tri-$\Var^!$-algebras.

We can prove the statement about an embedding of s-trialgebras into
corresponding Rota---Baxter algebras.

\begin{thm}
For every s-tri-$\Var$-dendriform algebra $A$ there
exists an algebra $\hat A\in \Var $
with a Rota---Baxter operator $R$ of weight zero
 and an injective map $\iota:A\to \hat A$ such that
$\iota (a\vdash_i b) = R(\iota(a))\circ_i \iota(b)$,
$\iota (a\dashv_i b) = \iota(a)\circ_i R(\iota(b))$,
and $\iota(a\perp_i b) = \iota(a)\circ _i \iota(b)$.
\end{thm}

\begin{proof}
To prove the statement, define Eilenberg construction for
$A$ as $\hat{A}=A\oplus A'$
by \eqref{eq:Hat-Defn}, but also with one difference:
$a\circ_{i}b=a\dashv_i b+a\vdash_i b$.
This is a Rota---Baxter algebra with an operator $R$ from Lemma \ref{lem:RB-dialg}.
Other steps of the proof of Theorem \ref{thm:DoubleEmbedding}
 remain the same.
\end{proof}

The work is supported by RFBR (project 09--01--00157),
Integration Project SB RAS No.94, and the
Federal Target Grant
``Scientific and educational staff of innovation Russia'' for 2009--2013
(contracts 02.740.11.5191, 02.740.11.0429, and 14.740.11.0346).


\begin{thebibliography}{99}

\bibitem{Aguiar00}
M.~Aguiar,
Pre-Poisson algebras, Lett. Math. Phys. 54 (2000) 263--277.

\bibitem{Hopf03}
G.~E.~Andrews, L.~Guo, W.~Keigher, K.~Ono,
Baxter algebras and Hopf algebras,
Trans. Amer. Math. Soc. 355 (2003) 4639--4656.

\bibitem{TernaryAss}
H. Ataguema, A. Makhlouf,
Notes on cohomologies of ternary algebras of associative type,
J. Gen. Lie Theory Appl. 3 (2009) no~3, 157--174.

\bibitem{BaiGuoNi10}
C.~Bai, L.~Guo, X.~Ni,
$\mathcal{O}$-operators on associative
algebras and dendriform algebras, arXiv:1003.2432v2 [math.RA] (March 2010).

\bibitem{BBGN2011}
C.~Bai, O.~Bellier, L.~Guo, X.~Ni,
Splitting of operations, Manin products, and Rota---Baxter operators,
arXiv:1106.6080v1 [math.QA] (June 2011).

\bibitem{BDK2001}
B.~Bakalov, A.~D'Andrea, V.~G.~Kac,
  Theory of finite pseudoalgebras,
  Adv. Math.  162 (2001) no.~1, 1--140.

\bibitem{Baxter60}
G.~Baxter,
An analytic problem whose solution follows from a simple
algebraic identity, Pacific J. Math. 10 (1960) 731--742.

\bibitem{BelaDrin82}
A.~A.~Belavin, V.~G.~Drinfeld,
Solutions of the classical Yang-Baxter equation for simple Lie algebras,
Funct. Anal. Appl. 16 (1982) 159--180.

\bibitem{BokutChen}
L.~A. Bokut, Y.~Chen, X.~Deng,
Gr\"obner-Shirshov bases for Rota-Baxter algebras,
Siberian Mathematical Journal 51 (2010) no.~6, 978--988.

\bibitem{Cartier72}
P.~Cartier,
On the structure of free Baxter algebras,
Advances in Math. 9 (1972) 253--265.

\bibitem{Chapoton}
F.~Chapoton,
Un endofoncteur de la cat\'egorie des op\'erades,
Dialgebras and related operads.
Springer-Verl., Berlin, 2001, pp.~105--110.
(Lectures Notes in Math., vol. 1763).

\bibitem{Chen11}
Y.~Chen, Q.~Mo,
Embedding dendriform algebra into its universal enveloping Rota-Baxter algebra,
Proc. Amer. Math. Soc., 2011.

\bibitem{QFT00}
A.~Connes, D.~Kreimer,
Renormalization in quantum field theory and the Riemann-Hilbert problem. I.
The Hopf algebra structure of graphs and the main theorem., Comm. Math. Phys.
210 (2000) no.~1, 249--273.

\bibitem{QFT01}
A.~Connes, D.~Kreimer,
Renormalization in quantum field theory and the Riemann-Hilbert problem. II.
The $L$-function, diffeomorphisms and the renormalization group.,
Comm. Math. Phys. 216 (2001) no.~1, 215--241.

\bibitem{Fard02}
K.~Ebrahimi-Fard,
Loday-type algebras and the Rota-Baxter relation,
Lett. Math. Phys. 61 (2002) 139--147.

\bibitem{NumThe06}
K.~Ebrahimi-Fard, L.~Guo,
Quasi-shuffles, Mixable Shuffles and Hopf Algebras,
J. Algebraic Combinatorics 24 (2006) 83--101.

\bibitem{FardGuo07}
K.~Ebrahimi-Fard, L.~Guo,
Rota---Baxter algebras and dendriform algebras,
Jour. of Pure and Appl. Algebra 212 (2008) no.~2, 320--339.

\bibitem{GinzKapr}
V.~Ginzburg, M.~Kapranov,
 Koszul duality for operads,
 Duke Math. J. 76 (1994) no.~1, 203--272.

\bibitem{Kac&Gol}
M.~I.~Golenishcheva-Kutuzova, V.~G.~Kac,
$\Gamma$-conformal algebras,
J. Math. Phys. 39 (1998) no.~4, 2290--2305.

\bibitem{GubKol}
V. Gubarev, P. Kolesnikov,
The Tits-Kantor-Koecher construction for Jordan dialgebras.
Comm. Algebra 39 (2011) no.~2, 497--520.

\bibitem{Kol2006b}
P. Kolesnikov,
Identities of conformal algebras and pseudoalgebras.
Comm. Algebra 34 (2006), no. 6, 1965--1979.

\bibitem{Kol2008}
P. Kolesnikov,
Varieties of dialgebras and conformal algebras (Russian),
Sib. Mat. Zh. 49 (2008) no.~2, 323--340.
arXiv:math/0611501v3 [math.QA] (August 2007).

\bibitem{Leinster}
T.~Leinster,
 Higher operads, higher categories,
 Cambridge: Cambridge University Press, 2004
 (London Mathematical Society Lecture Note Series, vol.~298).

\bibitem{Leroux04}
P.~Leroux,
Construction of Nijenhuis operators and dendriform trialgebras,
Int. J. Math. Math. Sci. 52 (2004) no.~40, 2595--2615.

\bibitem{Loday93}
J.-L.~Loday,
Une version non commutative des alg\'ebres de Lie:
les alg\'ebres de Leibniz,
Enseign. Math. 39 (1993) 269--293.

\bibitem{LP93}
J.-L.~Loday, T.~Pirashvili,
Universal enveloping algebras of Leibniz algebras and (co)homology,
Math. Ann., 296 (1993) 139--158.

\bibitem{Dialg99}
J.-L.~Loday,
Dialgebras,
Dialgebras and related operads.
Springer-Verl., Berlin, 2001, pp.~1--61.
(Lectures Notes in Math., vol. 1763).

\bibitem{Trialg01}
J.-L.~Loday, M.~Ronco,
Trialgebras and families of polytopes, in ``Homotopy Theory: Relations
with Algebraic Geometry, Group Cohomology, and Algebraic K-theory'',
Comtep. Math. 346 (2004) 369--398.

\bibitem{May}
J.-P. May,
 Geometry of iterated loop spaces,
 Springer-Verl., New York, 1972.
 (Lecture Notes in Math., vol.~271).

\bibitem{Pozh2009}
A.~Pozhidaev,
0-dialgebras with bar-unity, Rota-Baxter and 3-Leibniz
  algebras, in Groups, Rings and Group Rings,  ed. A. Giambruno
  et al. (Providence, RI: American Mathematical Society (AMS), 2009)
pp.~245--256.

\bibitem{Rota68}
G.-C.~Rota. Baxter algebras and combinatorial identities I, II,
Bull. Amer. Math. Soc. 75 (1969) 325--329, 330--334.

\bibitem{Semenov83}
M.~A.~Semenov-Tian-Shansky,
What is a classical r-matrix?
Funct. Anal. Appl., 17 (1983) no.~4, 259--272.

\bibitem{Spitz56}
F.~Spitzer,
A combinatorial lemma and its application to probability theory,
Trans. Amer. Math. Soc. 82 (1956) 323--339.

\bibitem{Stasheff}
J.~Stasheff,
What Is...an Operad?
Notices of the American Mathematical Society 51 (2004) no.~6, 630--631.

\bibitem{Uchino}
K. Uchino,
Derived bracket construction and Manin products,
Lett. Math. Phys. 90 (2010), 37--53.

\bibitem{Vallette2007}
B. Vallette,
Homology of generalized partition posets,
J. Pure Appl. Algebra 208 (2007) no.~2, 699--725.

\bibitem{Vallette2008}
B. Vallette,
Manin products, Koszul duality, Loday algebras and Deligne conjecture,
J. reine angew. Math. 620 (2008), 105--164.

\bibitem{Voronin}
V.~Yu.~Voronin,
Special and exceptional Jordan dialgebras, arXiv:1011.3683 [math.RA] (May 2010).

\bibitem{Zinbiel}
G.~W.~Zinbiel,
Encyclopedia of types of algebras 2010,
arXiv:1101.0267v1 [math.RA] (December 2010).


\end{thebibliography}
\end{document}